\documentclass[11pt]{amsart}
\usepackage[margin=1.25in]{geometry}

\usepackage{amsfonts,amsmath,amssymb,amsthm,graphics,subfig,enumerate,
color,cite,verbatim}

\newif \ifshowvc
\showvctrue
\showvcfalse
\ifshowvc
\input{vc}
\fi

\DeclareMathOperator{\GL}{GL}
\DeclareMathOperator{\SL}{SL}

\renewcommand{\tilde}{\widetilde}
\newcommand{\C}{\mathbb{C}}

\newcommand{\Z}{\mathbb{Z}}

\newcommand{\ra}{\rightarrow }

\newcommand{\cL}{\mathcal{L}}
\newcommand{\OO}{{\mathcal {O}}}

\newcommand{\p}{{\mathfrak {p}}}

\newcommand{\la}{\lambda }

\newcommand{\sgn}{{\mathrm {sgn}}}
\newcommand{\Dem}{{\mathcal {D}}}

\newcommand{\T}{{\mathcal {T}}}

\newcommand{\W}{{\mathcal {W}}}

\newcommand{\Ts}{\tilde{\mathcal {T}}}
\newcommand{\tDelta}{{\Delta _v}}
\newcommand{\RR}{\tilde{\mathcal {R}}}
\newcommand{\St}{\tilde{\mathcal {S}}}
\newcommand{\bG}{\mathbf{G}}

\newcommand{\A}{\tilde{\mathcal {A}}}
\newcommand{\polyring}{\mathcal {A}}
\newcommand{\ffield}{\mathcal {K}}

\newtheorem{lemma}{Lemma}
\newtheorem{prop}[lemma]{Proposition}
\newtheorem{theorem}[lemma]{Theorem}
\newtheorem{corollary}[lemma]{Corollary}

\theoremstyle{definition}
\newtheorem{remark}{Remark}

\author{Gautam Chinta, Paul E. Gunnells, Anna Pusk\'as}

\address{Department of Mathematics,
The City College of New York,
New York, NY 10031, USA
}

\email{chinta@sci.ccny.cuny.edu}

\address{Department of Mathematics and Statistics,
University of Massachusetts,
Amherst, MA 01003, USA
}
\email{gunnells@math.umass.edu}

\address{Department of Mathematics,
Columbia University,
New York, NY 10027, USA
}
\email{apuskas@math.columbia.edu}

\title{Metaplectic Demazure 
  operators and Whittaker functions}
\date{August 22, 2014}

\begin{document}
\begin{abstract}
In \cite{cg-jams} the first two named authors defined an action of a
Weyl group on rational functions and used it to construct multiple
Dirichlet series.  These series are related to Whittaker functions on
an $n$-fold metaplectic cover of a reductive group.  In this paper, we
define metaplectic analogues of the Demazure and Demazure-Lusztig
operators.  We show how these operators can be used to recover the
formulas from \cite{cg-jams}, and how, together with results of
McNamara \cite{mcnamara2}, they can be used to compute Whittaker
functions on metaplectic groups over $p$-adic fields.
\end{abstract}


\maketitle


\section{Introduction}

The Casselman-Shalika formula is an explicit formula for the values of
the spherical Whittaker functions associated to an unramified
principal series representation of a reductive group over a
non-archimedian local field $F$ \cite{casselman-shalika-cs2},
generalizing earlier work of Shintani \cite{shintani}.  This
has proven to be an important tool in the study of automorphic forms,
and in particular, in the construction of $L$-functions.  Similarly,
the \emph{metaplectic} Casselman-Shalika formula is relevant to the
study of certain Dirichlet series in several complex variables that
are expected to be the global Whittaker functions of Eisenstein series
on metaplectic covers of reductive groups.

Three related but distinct approaches to generalizing the
Casselman-Shalika formula to the nonlinear setting have recently
emerged.  The first is found in work of Brubaker-Bump-Friedberg
\cite{bbf-annals}.  Working over a global field and building on
earlier work with Hoffstein \cite{bbfh-wmd3}, these authors compute
the Whittaker functions of the Borel Eisenstein series on a
metaplectic cover of $\SL_r$.  A recursion relating Whittaker
functions on a cover of $\SL_{r}$ to those on $\SL_{r-1}$ plays a key
role in their proof.  They show that though the Whittaker functions
are not Euler products, they do satisfy a certain twisted
multiplicativity that reduces their specification to a description of
their $p$-parts, for $p$ a prime.  These $p$-parts are then shown to
be expressible in terms of sums over a crystal base.

Second, the work of McNamara \cite{mcnamara} treats metaplectic covers
$\tilde{G}$ of a simply-connected Chevalley group $G$ over a local
field.  He directly computes the spherical Whittaker function by
integrating the spherical vector $\varphi_{K} $ over the (opposite)
unipotent subgroup $U^{-}$.  McNamara defines a decomposition of
$U^{-}$ into a collection of disjoint subsets in bijection with the
(infinite) crystal graph $B (-\infty)$; on each subset the integrand
$\varphi_{K}$ is constant.  This proves that the Whittaker function
can be realized as a sum over a crystal base.  When $G=\SL_{r}$, he
recovers the formulas of Brubaker-Bump-Friedberg-Hoffstein.

Finally, a third approach appears in the work of Chinta-Offen
\cite{co-cs}.  This expresses the $p$-adic Whittaker functions on a
metaplectic cover of $\GL_r$ over a $p$-adic field as a sum over the Weyl
group.  This approach has since been generalized by McNamara
\cite{mcnamara2} to the context of tame covers of unramified reductive
groups over a local field.  The formulas in these works involve a
``metaplectic'' action of the Weyl group on rational functions.  This
action, which has its origins in Kazhdan-Patterson's seminal
investigation of automorphic forms on metaplectic covers of $\GL_{r}$
\cite{MR743816}, was used by two of us (GC and PG) to construct
\emph{Weyl group multiple Dirichlet series} \cite{cg-quadratic,
  cg-jams}.  These are infinite series in several complex variables
analogous to the classical Dirichlet series in one variable, such as
the Riemann $\zeta$ function and Dirichlet $L$-functions.  They
satisfy a group of functional equations isomorphic to the Weyl group
that intermixes the variables.  A consequence of the works
\cite{co-cs, mcnamara2} is that the $p$-adic metaplectic Whittaker
functions coincide with the local factors of these series
(cf.~\S\ref{sec:whittaker}).

It is the formulas arising in the third approach that concern us in
this article, which is partially motivated by connections between
Whittaker functions and the geometry and combinatorics of Schubert
varieties.  In the nonmetaplectic case, that Whittaker functions on
$G/F$ are related to the geometry of the flag variety $X$ attached to
the complex dual group $\hat G (\C)$ has been recently elucidated by
Brubaker-Bump-Licata \cite{bbl-demazure}, following earlier work of
Reeder \cite{reeder}.  In particular, recall that if $S\subset X$ is a
Schubert variety and $\cL$ is a line bundle on $X$ with global
sections, then the space $H^{0} (S, \cL)$ is a $\hat T (\C)$-module,
where $\hat T (\C)\subset \hat G (\C)$ is a maximal torus.  The
character of such a module is called a \emph{Demazure character}, and
can be computed by applying \emph{Demazure operators} to a highest
weight monomial \cite{demazure,andersen}.  Then Brubaker-Bump-Licata
prove (among other results) that the Iwahori Whittaker functions
become Demazure characters when $q^{-1}\rightarrow 0$, where $q$ is
the cardinality of the residue field.

To generalize results of \cite{bbl-demazure} to the metaplectic case,
a first step is developing a metaplectic analogue of the Demazure
character formula. (Demazure's version of the Weyl character formula
appears in \cite{Fulton_Young}.)  This is accomplished in the present
paper.  We define metaplectic Demazure and \emph{Demazure-Lusztig}
operators using the metaplectic Weyl group action found in
\cite{cg-jams, cg-quadratic}.  We prove that these operators satisfy
the same relations as their classical counterparts.  We also prove an
analogue of Demazure's version of the Weyl character formula
(corresponding to the case $S=X$ above) (Theorem \ref{thm:long_word}),
as well as a companion identity for the Demazure-Lustig operators
(Theorem \ref{thm:T_sum}) and show how they can be used to compute
spherical Whittaker functions on metaplectic covers (Theorem
\ref{thm:demazure_whittaker}).

\medskip \noindent \textbf{Acknowledgments.} We thank Dan Bump,
Cristian Lenart and Peter McNamara for helpful conversations.  We
thank the NSF for partially supporting this work through grants DMS
0847586 (GC) and DMS 1101640 (PG).

\section{Notation}\label{sec:notation} We begin by setting up
notation.  For unexplained notions about root systems and Coxeter
groups, we refer to \cite{bourbaki}.

Let $\Phi$ be an irreducible reduced root system of rank $r$ with Weyl
group $W$. Choose an ordering of the roots and let $\Phi = \Phi^+ \cup
\Phi^{-}$ be the decomposition into positive and negative roots.  Let
$\{\alpha_1, \alpha_2, \ldots, \alpha_r\}$ be the set of simple roots,
and let $\sigma_i$ be the Weyl group element corresponding to the
reflection through the hyperplane perpendicular to $\alpha_i$.  Define 
\begin{equation}
  \label{def:Phi_w} \Phi(w)=\{\alpha\in\Phi^+: w(\alpha)\in\Phi^-\}.
\end{equation}

Let $\Lambda$ be a lattice containing $\Phi$ as a subset.  Later
(Section \ref{sec:whittaker}) we will assume that $\Lambda$ is the
coweight lattice of a split reductive algebraic group $G$ defined over
the non-archimedean field $F$, and $\Phi$ is its coroots, but at the
moment this is not necessary.  Right now all we require is that the
Weyl group $W$ acts on $\Lambda$, and that there is a $W$-invariant
$\Z$-valued quadratic form $Q$ defined on $\Lambda$.  Define a
bilinear form $B (\alpha ,\beta)$ by $Q (\alpha +\beta) - Q (\alpha)-Q
(\beta)$.

We fix a positive integer $n$.  The integer $n$ determines a
collection of integers $\{m (\alpha) : \alpha \in \Phi\}$ by
\begin{equation}\label{eqn:defofm}
  m(\alpha)=n/\gcd(n, Q (\alpha )),
\end{equation}
and a sublattice $\Lambda_0\subset \Lambda$ by 
\begin{equation}\label{eqn:defoflambda0}
\Lambda_{0} = \{\lambda \in \Lambda : \text{$B (\alpha ,\lambda) \equiv 0
\bmod n$ for all simple roots $\alpha$}\}.
\end{equation}
With these definitions, one can easily prove the following:
\begin{lemma}
  For any simple root $\alpha$, we have $m(\alpha)\alpha\in
  \Lambda_0$. \qed
\end{lemma}

Let $\polyring=\C[\Lambda]$ be the ring of Laurent polynomials on
$\Lambda$ and $\ffield$ its field of
fractions. The action of $W$ on the lattice $\Lambda$
induces an action of $W$ on $\ffield$: we put
\begin{equation}
  \label{eq:W_on_K}
  (w, x^\lambda)\longmapsto x^{w\lambda} =:w.x^{\lambda },
\end{equation}
and then extend linearly and multiplicatively to all of $\ffield$.
We will always denote this
action using the lower dot 
\[
(w,f)\longmapsto w.f
\]
to distinguish it from the metaplectic $W$-action on $\ffield$
constructed below in \eqref{eq:metaWnotation}.  

Let $\lambda \mapsto \bar \lambda $ be the projection $\Lambda
\rightarrow \Lambda /\Lambda_{0}$ and $(\Lambda/\Lambda_0)^*$ be the
group of characters of the quotient lattice.
Any $\xi\in (\Lambda/\Lambda_0)^*$ induces a field isomorphism of
$\ffield/\C$ by setting $\xi(x^\lambda)=\xi(\bar\lambda)\cdot x^\lambda$
for $\lambda\in\Lambda.$ This leads to the direct sum decomposition
\begin{equation}
  \label{eq:direct_sum}
  \ffield=\bigoplus_{\bar\lambda\in \Lambda/\Lambda_0}
  \ffield_{\bar\lambda}
\end{equation}
where $\ffield_{\bar\lambda}=\{f\in\ffield: \xi(f)=\xi(\bar\lambda)\cdot f
\mbox{ for all $\xi\in(\Lambda/\Lambda_0)^* $}\}$

Next choose nonzero complex parameters $v,g_0, \ldots, g_{n-1}$
satisfying
\begin{equation}
  \label{eq:qt}
  g_0=-1 \mbox{  and  } g_ig_{n-i}=v^{-1}\mbox{  for  } i=1,\ldots, n-1;
\end{equation}
for all other $j$ we define $g_{j}:=g_{r_{n} (j)}$,
where $0\leq r_{n} (j)<n-1$ denotes the remainder upon dividing $j$ by $n$.
Introduce the following deformation of the Weyl denominator:
\begin{equation*}
  \tDelta =\prod _{\alpha \in \Phi ^{+}} 
  \bigl(1-v\cdot x^{m(\alpha )\alpha }\bigr).
\end{equation*}
If $v=1$ we write more simply $\Delta_v=\Delta.$

We now define an action of the Weyl group $W$ on $\ffield$ as
follows.  For $f\in \ffield_{\bar\lambda}$ and $\sigma_\alpha\in W$ the
generator corresponding to a simple root $\alpha$, define

\begin{equation}
  \label{def:action}
  \begin{split}
\sigma _i(f)=\frac{\sigma _i. f}{1-vx^{m(\alpha _i)\alpha _i}}&\cdot 
\left[x^{-r_{m(\alpha _i)}\left(-\frac{B(\la,\alpha _i)}{Q(\alpha
        _i)}\right)\cdot \alpha _i}\cdot (1-v)\right.\\
&\left. {}-v\cdot g_{Q(\alpha _i)-B(\la ,\alpha _i)}\cdot 
x^{(1-m(\alpha _i))\alpha _i}\cdot (1-x^{m(\alpha _i)\alpha _i})\right]    
  \end{split}
\end{equation}
where $\lambda$ is any lift of $\bar \lambda$ to $\Lambda$.  It is
easy to see that the quantity in brackets depends only on $\bar
\lambda$.  We extend the definition of $\sigma_\alpha$ to $\ffield$ by
additivity.  One can check that with this definition,
$\sigma^2_\alpha(f)=f$ for all $f\in \ffield$.  Furthermore it is
proven in \cite{cg-jams} (see also \cite{mcnamara2}) that this action
satisfies the defining relations of $W$: if $(m_{i,j})$ is the Coxeter
matrix for $\Phi$, then
\begin{equation}
  \label{eq:W_braid}
(\sigma_{i}\sigma_{j})^{m_{i,j}} (f) = f\quad \text{for all $i,j$ and $f\in \ffield$.}
\end{equation}
Therefore (\ref{def:action}) extends to an action of the full Weyl
group $W$ on $\ffield$, which we denote 
\begin{equation}\label{eq:metaWnotation}
(w,f) \longmapsto w (f).
\end{equation}
We remark that if $n=1$, the action \eqref{def:action} collapses to
the usual action \eqref{eq:W_on_K} of $W$ on $\ffield $.  That the
quantity in brackets in \eqref{def:action} depends only on $\bar
\lambda$ and not $\lambda$ translates to the following lemma:

\begin{lemma}
  \label{lemma:h_exchange}
  Let $f\in \ffield$ and $h\in \ffield_0$.  Then for any $w\in W$,
$$w(hf)=(w.h)\cdot w(f).$$
Here $w.h$ means the action of \eqref{eq:W_on_K}, whereas $\cdot$
denotes multiplication in $\ffield$.\hfill \qed
\end{lemma}

Lemma \ref{lemma:h_exchange} is used repeatedly in the proofs
below. It is important to note that the action of $W$ on $\ffield$
defined by \eqref{def:action} is $\C$-linear, but is {\em{not}} by
endomorphisms of that ring, i.e.~it is not in general multiplicative.
The point of Lemma \ref{lemma:h_exchange} is that if we have a product
of two terms $hf$, the first of which satisfies $h\in \ffield_{0}$, then
in \eqref{def:action} we can apply $w$ to the product $hf$ by
performing the usual permutation action on $h$ and then acting on $f$
by the twisted $W$-action.

Next we use this Weyl group action to define certain divided
difference operators.  For $1\leq i\leq r$ and $f\in \ffield$ define
the \emph{Demazure operators} by
\begin{equation}\label{def:dem}
  \Dem_i(f)=\Dem_{\sigma _i}(f)=
  \frac{f-x^{m(\alpha _i)\alpha _i}\cdot \sigma _i(f)}
  {1-x^{m(\alpha _i)\alpha _i}},
\end{equation}
and the \emph{Demazure-Lusztig operators} by
\begin{equation}\label{def:dl}
  \begin{split}
    \T _i(f)=\T _{\sigma _i}(f)&=
    \left(1-v\cdot x^{m(\alpha _i)\alpha _i}\right)\cdot \Dem _i(f) -f\\
    &=\left(1-v\cdot x^{m(\alpha _i)\alpha _i}\right)\cdot
    \frac{f-x^{m(\alpha _i)\alpha _i}\cdot \sigma _i(f)}
    {1-x^{m(\alpha _i)\alpha _i}}-f.
  \end{split}
\end{equation}
When there is no danger of confusion, we write more simply
\begin{equation*}
  \Dem_i=
  \frac{1-x^{m(\alpha _i)\alpha _i}\sigma _i}
  {1-x^{m(\alpha _i)\alpha _i}}\mbox{  and  }
  \T_i=\left(1-v\cdot x^{m(\alpha _i)\alpha _i}\right)\cdot \Dem_i - 1,
\end{equation*}
that is, a rational function $h$ in the above equations is interpreted
to mean the ``multiplication by $h$'' operator. The rational functions
here are in $\ffield _0.$ 

We prove in the following section that the operators $\Dem _i$ and
$\T_i$ satisfy the same braid relations as the $\sigma _i.$
Consequently, we can define $\Dem_w$ and $\T_w$ for any $w\in W$ as
follows. Let $w=\sigma_{i_1}\cdots \sigma _{i_l}$ be a reduced
expression for $w$ in terms of simple reflections. Then we define
$$\Dem _w=\Dem _{i_1}\cdots \Dem _{i_l}
\mbox{\ \ \ \ and\ \ \ \ } \T _w=\T _{i_1}\cdots \T _{i_l}.$$

In the first two theorems below, both sides of the equalities are to
be understood as identities of operators on $\ffield.$

\begin{theorem}\label{thm:long_word}
  For the long element $w_0$ of the Weyl group $W$ we have
$$\Dem _{w_0} =
\frac{1}{\Delta } \cdot \sum _{w\in W} \sgn (w)\cdot \prod _{\alpha
  \in \Phi(w^{-1})} x^{m(\alpha )\alpha }\cdot w .$$
\end{theorem}

\begin{theorem}\label{thm:T_sum}
  We have
$$\tDelta \cdot \Dem _{w_0} =\sum _{w\in W} \T _w.$$
\end{theorem}

We prove Theorem 3 in Section \ref{sec:pf_long} and Theorem 4 in
Section \ref{sec:pf_T_sum}.

\begin{remark}
  In Section \ref{sec:whittaker} we use the work of McNamara
  \cite{mcnamara2} to express Whittaker functions over a local
  $p$-adic field in terms of the
  operators introduced above.  In this section the parameteters will
  be specialized:  $v$ will be set to equal $q^{-1}$ (for $q$ the
  cardinality of the residue field) and the $g_i$ will be Gauss sums.
  For now, we only need these parameters to satisfy the relations
  (\ref{eq:qt}). 
\end{remark}

\section{Basic properties of the operators}\label{sec:props}

In this section we prove the \emph{quadratic relations} (Proposition
\ref{prop:quadratic_properties}) and \emph{braid relations}
(Proposition \ref{prop:demazbraid}) satisfied by the Demazure and
Demazure-Lusztig operators.

\begin{prop}\label{prop:quadratic_properties}
  The operators $\Dem _i$ and $\T_i$ ($1\leq i\leq r$) satisfy the
  following quadratic relations:
  \begin{enumerate}[(i)]
  \item $\Dem _i ^2=\Dem _i;$
  \item $\T _i ^2=(v-1)\T _i +v.$
  \end{enumerate}
\end{prop}

\begin{proof}
We prove (i) in detail and leave (ii) to the reader.  To simplify the
notation, we drop the subscripts and write $\Dem , \alpha,$ and
$\sigma$, and abbreviate $m (\alpha)$ to $m$.  Using the
definition of $\Dem$ and Lemma \ref{lemma:h_exchange}, we
have \begin{align*} \Dem^2&= \Bigl(\frac{1-x^{m \alpha}\sigma}{1-x^{m\alpha
}}\Bigr)^2\\
    &= \Bigl(\frac{1}{(1-x^{m\alpha})^2}+
      \frac{x^{m\alpha}}{1-x^{m\alpha}}
      \cdot \frac{x^{-m\alpha}}
      {1-x^{-m\alpha}}\Bigr) \cdot 1 \\
    &\qquad\qquad+\Bigl(\frac{-x^{m\alpha}}
      {(1-x^{m\alpha})^2}+ \frac{-x^{m\alpha}}{1-x^{m\alpha}}
      \cdot \frac{1}{1-x^{-m\alpha}}\Bigr)\cdot \sigma \\
    &=\frac{1}{1-x^{m\alpha}}\cdot
    \Bigl(\frac{1}{1-x^{m\alpha}}+
      \frac{1}{1-x^{-m\alpha}}\Bigr)\cdot 1\\
    &\qquad \qquad+\frac{-x^{m\alpha}}{1-x^{m\alpha}} \cdot \Bigl(\frac{1}{1-x^{m\alpha}}+ \frac{1}{1-x^{-m\alpha}}\Bigr)\cdot
    \sigma.
  \end{align*}
Since
$$\frac{1}{1-x^{m\alpha}}+
\frac{1}{1-x^{-m\alpha}} =1$$ we obtain $\Dem
^2=\Dem$.
\end{proof}

We pause to point out the key role played by Lemma
\ref{lemma:h_exchange} in the proof of Proposition
\ref{prop:quadratic_properties}: the action of $\sigma_{i}$ on an
arbitrary rational function is given by the complicated formula
\eqref{def:action}, but thanks to Lemma \ref{lemma:h_exchange} we can
pass the operator $\sigma$ past the monomial $x^{m (\alpha)\alpha}$,
after acting on this monomial by the usual permutation action.  This
fact will be used repeatedly throughout the paper.

\begin{lemma} \label{lemma:quadratic} We have $\Dem _i x^{m(\alpha
    _i)\alpha _i}\Dem _i =-\Dem _i.$
\end{lemma}

\begin{proof}
We use the same notation as in the proof of Proposition
\ref{prop:quadratic_properties} and compute directly:
\begin{align*}
\Dem x^{m\alpha}\Dem &= \Dem \left( \frac{x^{m\alpha }-x^{2m\alpha}\sigma }{1-x^{m\alpha }}\right)\\
     		     &= \frac{x^{m\alpha }-x^{2m\alpha}\sigma }{(1-x^{m\alpha })^{2}}-\frac{\sigma -x^{-m\alpha}}{(1-x^{m\alpha }) (1-x^{-m\alpha })}\\
		     &= \frac{x^{m\alpha }\sigma - 1}{1-x^{m\alpha }}\\
		     &=-\Dem .
\end{align*}
\end{proof}

\begin{prop}\label{prop:demazbraid}
Suppose $(\sigma_{i}\sigma_{j})^{m_{i,j}}=1$ is a defining relation
for $W$.  Then 
\begin{align}
\label{eq:db1}\Dem_{i}\Dem_{j}\Dem_{i}\dotsb &= \Dem_{j}\Dem_{i}\Dem_{j}\dotsb, \\
\label{eq:db2}\T_{i}\T_{j}\T_{i}\dotsb &= \T_{j}\T_{i}\T_{j}\dotsb, 
\end{align}
where there are $m_{i,j}$ factors on both sides of \eqref{eq:db1}--\eqref{eq:db2}.
\end{prop}

\begin{proof}
Both statements boil down to explicit computations with rank $2$ root
systems, and in fact are special cases of Theorems \ref{thm:long_word}
and \ref{thm:T_sum}.  We explain what happens in detail with \eqref{eq:db1} in
$A_{2}$, which is typical of all the computations.  Since all roots have the
same length, we lighten notation by putting $m=m (\alpha)$.

By definition,
\[
\Dem_{1} = \frac{1-x^{m\alpha_{1}}\sigma_{1}}{1-x^{m\alpha_{1}}}.
\]
Next we apply $\Dem_{2}$ and use $\sigma_{2} (\alpha_{1})=\alpha_{1}+\alpha_{2}$:
\[
\Dem_{2}\Dem_{1} = \frac{1-x^{m\alpha_{1}}\sigma_{1}}{(1-x^{m\alpha_{2}}) (1-x^{m\alpha_{1}})}-\frac{x^{m\alpha_{2}}\sigma_{2}-x^{m (\alpha_{1}+2\alpha_{2})}\sigma_{2}\sigma_{1}}{(1-x^{m\alpha_{2}}) (1-x^{m(\alpha_{1}+\alpha_{2})})}
\]
Finally we apply $\Dem_{1}$ to obtain
\begin{multline*}
\Dem_{1}\Dem_{2}\Dem_{1} =
\frac{1-x^{m\alpha_{1}}\sigma_{1}}{(1-x^{m\alpha_{2}})
(1-x^{m\alpha_{1}})^{2}}-\frac{x^{m\alpha_{2}}\sigma_{2}-x^{m
(\alpha_{1}+2\alpha_{2})}\sigma_{2}\sigma_{1}}{(1-x^{m\alpha_{1}}) (1-x^{m\alpha_{2}})
(1-x^{m(\alpha_{1}+\alpha_{2})})} \\
-\biggl(\frac{x^{m\alpha_{1}}\sigma_{1}-1}{(1-x^{m\alpha_{1}}) (1-x^{m(\alpha_{1}+\alpha_{2})}) (1-x^{-m\alpha_{1}})}-\frac{x^{m(2\alpha_{1}+\alpha_{2})}\sigma_{1}\sigma_{2}-x^{m (2\alpha_{1}+2\alpha_{2})}\sigma_{1}\sigma_{2}\sigma_{1}}{(1-x^{m\alpha_{1}}) (1-x^{m\alpha_{2}}) (1-x^{m(\alpha_{1}+\alpha_{2})})} \biggr),
\end{multline*}
which simplifies to 
\begin{multline}\label{eq:final}
\Dem_{1}\Dem_{2}\Dem_{1} \\
=\frac{1-x^{m\alpha_{1}}\sigma_{1}-x^{m\alpha_{2}}\sigma_{2} + x^{m(2\alpha_{1}+\alpha_{2})}\sigma_{1}\sigma_{2}+x^{m (\alpha_{1}+2\alpha_{2})}\sigma_{2}\sigma_{1}-x^{m (2\alpha_{1}+2\alpha_{2})}\sigma_{1}\sigma_{2}\sigma_{1}}{\Delta },
\end{multline}
where $\Delta = (1-x^{m\alpha_{1}}) (1-x^{m\alpha_{2}}) (1-x^{m
(\alpha_{1}+\alpha_{2})})$.  The final formula \eqref{eq:final} clearly
depends only on the longest word in the Weyl group for $A_{2}$ and not
on the reduced expression used to define it, which proves
\eqref{eq:db1}. (Note that this computation also checks Theorem
\ref{thm:long_word} for $\Phi = A_{2}$.)
\end{proof}

\section{Proof of Theorem \ref{thm:long_word}}\label{sec:pf_long}

We now turn to the proof of Theorem \ref{thm:long_word}.  Before we
can begin, we require more notation.  The following is
\cite[Proposition 21.10]{Bump_LieGroups}, applied to $\Phi (w^{-1})$
instead of $\Phi (w)$:

\begin{prop}\label{prop:phi(w)description}
  Let $w=\sigma _{i_1}\sigma _{i_2}\cdots \sigma_{i_N}$ be a reduced
  expression for $w\in W$.  Then the set
$$\Phi(w^{-1})=\{\alpha \in \Phi^+: w^{-1}(\alpha )\in \Phi ^-\}$$ 
consists of the elements
$$\alpha_{i_{1}},\ \sigma_{i_1}(\alpha_{i_{2}}),\
\sigma_{i_1}\sigma_{i_2}(\alpha_{i_{3}}) ,\ldots ,\ \sigma
_{i_1}\cdots \sigma _{i_{N-1}} (\alpha _{i_{N}}),$$ where the
$\alpha_{i}$ are the simple roots.
\end{prop}

Let $p\colon \Phi \ra \ffield_0$ be a map.  We say $p$ is
\emph{$W$-intertwining} if for any $\beta \in \Phi$ and $w\in W$, we have 
$$p(w\beta )=w.p(\beta ).$$ Proposition
\ref{prop:phi(w)description} has the following corollary, useful for
the proof of Theorems \ref{thm:long_word} and \ref{thm:T_sum}:

\begin{corollary}\label{corollary:rearrangepolys}
  Assume $p\colon \Phi \ra \ffield_0$ is $W$-intertwining, and suppose
$w\in W$ has a reduced expression $w=\sigma _{i_1}\sigma _{i_2}\cdots
\sigma _{i_N}$. Then we have the following equality of operators on
$\ffield$:
\begin{equation}\label{eq:lhs}
p(\alpha _{i_1})\sigma_{i_1}\cdot p(\alpha _{i_2})\sigma _{i_2}
\cdots p(\alpha _{i_N})\sigma _{i_N}
=\biggl(\prod _{\alpha \in \Phi(w^{-1})} p(\alpha )\biggr) \cdot w.
\end{equation}
\end{corollary}
\begin{proof}
  Making repeated use of Lemma \ref{lemma:h_exchange}, we can re-order
the operators on the left of \eqref{eq:lhs} by passing all the
$\sigma_{i_j}$s to the right and all elements of $\ffield_0$ to the
left.  After this, the left of \eqref{eq:lhs} becomes
$$
p(\alpha _{i_1})\cdot \left(\sigma _{i_1}.p(\alpha _{i_2})\right)
\cdot \left(\sigma _{i_1}\sigma _{i_2}.p(\alpha _{i_3})\right)
\cdots \left(\sigma_{i_1}\cdots \sigma_{i_{n-1}}.p(\alpha_{i_N})\right)
\cdot \sigma_{i_1}\sigma_{i_2}\cdots \sigma_{i_N}.
$$
Here $\sigma _{i_1}\dotsb \sigma _{i_N}$ is a reduced
expression for $w$. Moreover, by
Proposition \ref{prop:phi(w)description},
$$\alpha _{i_1},\ \sigma _{i_1}(\alpha _{i_2}),\ 
\sigma _{i_1}\sigma _{i_2}(\alpha _{i_3}),\ldots ,\ 
\sigma _{i_1}\cdots \sigma _{i_{N-1}}(\alpha _{i_N})$$
enumerates $\Phi(w^{-1}).$ As a consequence the corresponding
elements of $\ffield_0$, namely
$$p(\alpha _{i_1}),\ p(\sigma _{i_1}(\alpha _{i_2})),\ 
p(\sigma _{i_1}\sigma _{i_2}(\alpha _{i_3})),\ldots ,\ 
p(\sigma _{i_1}\cdots \sigma _{i_{N-1}}(\alpha _{i_N})),$$
have product
$\displaystyle \prod _{\alpha \in \Phi(w^{-1})}p(\alpha ).$ 
Since the map $p$ is $W$-intertwining, these are exactly
the factors appearing on the left of \eqref{eq:lhs}.
\end{proof}

We now begin the proof of Theorem \ref{thm:long_word}.  First notice
that by Lemma $\ref{lemma:h_exchange},$ any composition of the
operators $\Dem_i$ can be written as a $\ffield_0$-linear combination
of the operators $w\in W.$ Hence we can write
\begin{equation}
  \label{eq:del_long}
  \Dem_{w_0}=\sum_{w\in W}R_w\cdot w,
\end{equation}
for some choice of rational functions $R_w\in \ffield_0.$ 

Let $l\colon W\rightarrow \Z$ denote the length function on $W$.  It
is a standard fact about finite Coxeter groups that for any $1\leq
j\leq r,$ we can find a reduced expression $\sigma _{i_1}\sigma
_{i_2}\ldots \sigma _{i_{l(w_0)}}$ for the longest word $w_{0}$ with
$i_1=j.$ By Proposition \ref{prop:demazbraid}, we have $\Dem
_{w_0}=\Dem _j \Dem _{w_j}$ for $w_j=\sigma _{i_2}\ldots \sigma
_{i_{l(w_0)}}$.  Since $\Dem_{j}^{2}=\Dem_{j}$ (Proposition
\ref{prop:quadratic_properties}), we have
$\Dem_j\Dem_{w_0}=\Dem_{w_0}$.  In other words,
$$\frac{1-x^{m(\alpha_j)\alpha_j}\sigma_j}
{1-x^{m(\alpha_j)\alpha_j}}\Dem_{w_0}=\Dem_{w_0}. $$
It follows that $\sigma_j\Dem_{w_0}=\Dem_{w_0}.$ 
Now apply $\sigma_j$ to both sides of (\ref{eq:del_long}).  Since each
$R_w\in \ffield_0$, Lemma \ref{lemma:h_exchange} 
implies
\begin{equation}
  \label{eq:del_long2}
  \Dem_{w_0}=\sum_{w\in W}(\sigma_j.R_w)\cdot \sigma_jw.
\end{equation}
Comparing coefficients in \eqref{eq:del_long} and \eqref{eq:del_long2}
and using the fact that the elements of $W$ are linearly independent as
operators on $\ffield$, we obtain $\sigma_j .R_w=R_{\sigma_jw}$.  Thus 
\begin{equation}
  \label{eq:v_modify}
u.R_w=R_{uw} \hskip .5 cm \forall u,w\in W
\end{equation}

To finish the proof of Theorem \ref{thm:long_word}, it suffices to
compute $R_{w_0}$; the remaining coefficients can then be computed
using \eqref{eq:v_modify}. In fact we shall prove the following:
\begin{lemma}
  \label{lemma:Rw} For $w\in W$, we have 
$$R_w= 
\frac{\sgn(w)}{\Delta }\prod_{\alpha\in\Phi(w^{-1})} x^{m(\alpha)\alpha}.$$
\end{lemma}

\begin{proof}
  To start, assume the statement is true for $w=w_0$:
\begin{equation}
 \label{eq:Rw_0}
R_{w_0}=
\frac{\sgn(w_0)}{\Delta }\prod_{\alpha\in\Phi^+} x^{m(\alpha)\alpha}.
\end{equation}
 By \eqref{eq:v_modify}, we have
\begin{align*}
  R_{uw_0}=u.R_{w_0}&=u.\biggl(\sgn (w_0)\cdot 
\prod_{\alpha \in \Phi ^{+}}\frac{ x^{m(\alpha )\alpha }}
{1-x^{m(\alpha )\alpha }}\biggr) \\
&=\sgn (w_0)\cdot \prod _{\alpha \in u(\Phi^{+})}
\frac{x^{m(\alpha )\alpha }}{1-x^{m(\alpha )\alpha }}\\
&=R_{w_0}\sgn(u)\prod_{\alpha\in\Phi(u^{-1})}x^{-m(\alpha)\alpha}
\end{align*}
Let $u=ww_0.$  Then $\Phi(u^{-1})=\Phi^+\cap w(\Phi^+).$
Hence 
\begin{align*}
  R_{w}&=R_{w_0}\sgn(ww_0) \prod _{\alpha \in \Phi ^+ \cap w(\Phi
    ^{+})} x^{-m(\alpha )\alpha } \\
&= \frac{\sgn(w)}{\Delta}
\prod_{\alpha \in \Phi^+ } x^{m(\alpha )\alpha }
\prod _{\alpha \in \Phi ^+ \cap w(\Phi^{+})} x^{-m(\alpha )\alpha}\\
&= \frac{\sgn(w)}{\Delta}\prod_{\alpha \in \Phi(w^{-1})} x^{m(\alpha )\alpha}.
\end{align*}
Thus the proof will be complete if we show \eqref{eq:Rw_0}.

Begin by writing
\begin{equation*}
  \Dem_i=p_1(\alpha_i)-p_2(\alpha_i)\sigma_i,
\end{equation*}
where $p_1,p_2\colon \Phi\to \ffield _0$   are defined by
\begin{equation*}
  p_1(\beta)=\frac {1}{1-x^{m(\beta)\beta}},  \ \ \  
p_2(\beta)=\frac {x^{m(\beta)\beta}}{1-x^{m(\beta)\beta}}.
\end{equation*}
Given a reduced expression $w_0=\sigma _{i_1}\sigma _{i_2} \cdots
\sigma _{i_N}$, it is easy to see that
\[
R_{w_0}w_0=\sgn(w_0)p_2(\alpha _{i_1})\cdot \sigma _{i_1}\cdot p_2(\alpha _2)\cdot \sigma _{i_2} \cdots p_2(\alpha _{i_N})\cdot \sigma_{i_N}. 
\]
Since the map $p_2$ is readily shown to be $W$-intertwining, it
follows from the above equality and Corollary
\ref{corollary:rearrangepolys} that  
\[
R_{w_0} = \sgn(w)\prod_{\alpha\in\Phi^+}p_2(\alpha)=
\frac{\sgn(w_0)}{\Delta }\prod_{\alpha\in\Phi^+} x^{m(\alpha)\alpha}.
\]
This completes the proof of the lemma, and thus of Theorem
\ref{thm:long_word}.
\end{proof}

\section{Proof of Theorem \ref{thm:T_sum}} \label{sec:pf_T_sum} 

In this section we prove Theorem \ref{thm:T_sum}:
\begin{equation*}
  \tDelta \cdot \Dem _{w_0} =\sum _{w\in W} \T _w.
\end{equation*}
Let $\Ts=\sum _{w\in W} \T _w.$
  We begin with some lemmas.

\begin{lemma}\label{WW_eigenvalue}
 For any $1\leq i\leq r$ we have 
$$\T_i\cdot (\tDelta \Dem _{w_0}) =v\cdot (\tDelta \Dem _{w_0}) $$
\end{lemma}
\begin{proof}
Since the simple reflection $\sigma _i$ permutes the elements of $\Phi
^+\setminus \{\alpha _i\},$ the operator $\Dem_i$ commutes with
$$\prod _{\beta \in \Phi ^+\setminus \{\alpha _i\}}(1-vx^{m(\beta )\beta })=\frac{\tDelta}{1-vx^{m(\alpha _i)\alpha _i}}.$$
Consequently,
\begin{equation}
  \label{eq:3a}
  (1-vx^{m(\alpha _i)\alpha _i})\Dem _i \tDelta =
\tDelta \Dem_i (1-vx^{m(\alpha _i)\alpha _i}).
\end{equation}
Take a reduced expression for the long element,
$w_0=\sigma_{i_1}\sigma_{i_2}\cdots \sigma_{i_N}$ satisfying $i_1=i.$
Thus $\Dem_{w_0}=\Dem_i\Dem_{w_i}$ for $w_i=\sigma_{i_2}\cdots
\sigma_{i_N}.$ Using this and (\ref{eq:3a}), 
\[
(\T_i+1)\cdot (\tDelta \Dem_{w_0}) = 
\tDelta\Dem_i (1-vx^{m(\alpha _i)\alpha _i})\Dem _i \Dem_{w_i}.
\]
The idempotency of $\Dem _i$ (Proposition
\ref{prop:quadratic_properties}) and Lemma \ref{lemma:quadratic} imply 
\[
\Dem_i (1-vx^{m(\alpha _i)\alpha _i})\Dem _i=(1+v)\Dem_i.
\]
Putting everything together, we conclude that $(\T_i+1)\cdot (\tDelta
\Dem_{w_0}) = (1+v)\cdot (\tDelta\Dem_{w_0}) .$
\end{proof}

\begin{lemma}\label{Ts_eigenvalue}
 For any $1\leq i\leq r$ we have 
$$\T_i\cdot \Ts =v\cdot \Ts .$$
\end{lemma}
\begin{proof}
Recall that $l\colon W\rightarrow \Z$ is the length function on $W$.
Then for any element $w\in W$ and any simple reflection $\sigma _i$,
we have $l(\sigma _iw)=l(w)\pm 1$.  Partition $W$ into $C_1\cup C_2$,
where
\begin{align*}
C_1&=\{w\in W : l(\sigma _iw)=l(w)-1\},\\
C_2&=\{w\in W : l(\sigma _iw)=l(w)+1\}.
\end{align*}
Then the map $w\mapsto \sigma_iw$ defines a bijection between $C_1$
and $C_2.$  We compute
\begin{align*}
\T_i \Ts&=\T_i\cdot \biggl(\sum_{w\in C_1}\T _w +\sum _{w\in C_2}
\T _{w}\biggr)\\
&=\T_i\cdot \biggl(\sum_{w\in C_2}\T_i\T_w +\sum _{w\in C_2}\T
  _{w}\biggr)\\
&=\sum_{w\in C_2}\T_i^2\T_w +  \sum_{w\in C_2}\T_i\T_{w}.
\end{align*}
The second sum above is simply $\sum_{w\in C_1}\T _w.$ In the first,
we use the quadratic relation  $\T_i^2=(v-1)\T_i +v$ of Proposition
\ref{prop:quadratic_properties} to write 
\begin{equation*}
  \sum_{w\in C_2}\T_i^2\T_w=(v-1)\sum_{w\in C_2}\T_i\T_w 
+v\sum_{w\in C_2}\T_w= (v-1)\sum_{w\in C_1}\T_w +v\sum_{w\in C_2}\T_w.
\end{equation*}
Thus $\T_i\cdot \Ts =v\cdot \Ts .$
\end{proof}

\begin{lemma}\label{coeff_rel_for_eigenvectors}
Let 
$$\RR =\sum _{w\in W} R_w\cdot w$$
be an operator on $\ffield$ that is a linear combination of the Weyl group
elements with coefficients $R_w\in \ffield_0.$ Assume $\RR$ is an
eigenclass for $\T_{i}$ with eigenvalue $v$:
\begin{equation*}\label{eq:eigencondition_R}
 \T _i\cdot \RR =v\cdot \RR.
\end{equation*}
Then for every $w\in W$, we have
\begin{equation*}\label{eq:coeffrel}
R_w=\frac{1-vx^{m(\alpha_i)\alpha_i}}{1-vx^{-m(\alpha_i)\alpha_i}}\cdot 
\sigma _i.R_{\sigma _iw}.
\end{equation*}

\end{lemma}
\begin{proof}
The proof is a straightforward computation.  Begin with
\begin{equation*}
 \T_iR_w w= \left[ q_1(\alpha_i)-q_2(\alpha_i)\sigma_i \right]R_ww=
q_1(\alpha_i)R_ww-q_2(\alpha_i)(\sigma_i .R_w)\sigma_iw
\end{equation*}
with $q_1, q_2$ defined by
\begin{equation}\label{def:q1q2}
  q_1(\beta)=\frac {1-vx^{m(\beta)\beta}}{1-x^{m(\beta)\beta}}-1,
\mbox{\ \ \ \ \ and \ \ \  \ }  
q_2(\beta)=\frac {(1-vx^{m(\beta)\beta})x^{m(\beta)\beta}}
{1-x^{m(\beta)\beta}}.
\end{equation}
Summing over $w\in W$,
we get
\begin{equation*}
  \T_i\cdot \RR=\sum_{w\in W}\left[ q_1(\alpha_i)R_w-
q_2(\alpha_i)\sigma_i.R_{\sigma_i w} \right]w
\end{equation*}
But we also have $\T _i\cdot \RR =v\cdot \RR$, so comparing
coefficients yields
\begin{equation*}
   q_1(\alpha_i)R_w-q_2(\alpha_i)\sigma_i.R_{\sigma_i w} =
vR_w.
\end{equation*}
Solving for $R_w$ completes the proof.
\end{proof}

Lemma \ref{coeff_rel_for_eigenvectors} has the following easy and useful corollary.

\begin{corollary}\label{longword_eigenvect_determines}
 Let 
$$\RR =\sum _{w\in W} R_w\cdot w,\hskip .5 cm \St =\sum _{w\in W} S_w\cdot w$$
be two operators on $\ffield$ that are linear combinations of the
Weyl-group elements with coefficients $R_w,S_w\in \ffield_0.$ 
Assume that
\begin{equation*}\label{eq:eigencondition}
 \T _i\cdot \RR =v\cdot \RR,\hskip .5 cm \T_i\cdot \St =v\cdot \St
\end{equation*}
for every $i, 1\leq i\leq r$.  Assume further that we have
$R_{w_0}=S_{w_0}$ for the long element $w_0\in W$.  Then $\RR =\St $
as operators on $\A .$
\end{corollary}
\begin{proof}
 We show that $R_w=S_w$ for every $w\in W.$ This can be seen by descending induction on the length of $w.$ For $l(w)$ maximal we have $R_{w_0}=S_{w_0}$ by assumption. Now assume $l(\sigma _iw)=l(w)+1,$ and $R_{\sigma _iw}=S_{\sigma _iw}.$ It follows from Lemma \ref{coeff_rel_for_eigenvectors} that 
$$R_w=\frac{1-vx^{m(\alpha _i)\alpha _i}}
{1-vx^{-m(\alpha _i)\alpha _i}}\cdot \sigma_i.R_{\sigma_iw}=
\frac{1-vx^{m(\alpha _i)\alpha _i}}
{1-vx^{-m(\alpha _i)\alpha _i}}\cdot \sigma_i.S_{\sigma_iw}=S_w,$$
thus $R_w=S_w.$ This completes the proof. 
\end{proof}

We now turn to the proof of Theorem \ref{thm:T_sum}. 
Applying Lemmas \ref{WW_eigenvalue} and \ref{Ts_eigenvalue} to the operators $\tDelta
\Dem _{w_0}$ and $\Ts$, we have 
$$\T_i\cdot (\tDelta \Dem _{w_0}) =v\cdot \tDelta \Dem _{w_0} ,
\hskip .5 cm \T_i\cdot \Ts =v\cdot \Ts $$ for every $1\leq i\leq r.$
It follows from the definitions that as operators on $\ffield,$ both
$\tDelta \Dem _{w_0}$ and $\Ts$ can be written as a linear combination
of elements of $W$ with coefficients in $\ffield_0.$ Let us write
$$\tDelta \Dem _{w_0}=\sum _{w\in W} R_w\cdot w
\mbox{\ \ \ \ \ and\ \ \ \ \ }\Ts=\sum _{w\in W} S_w\cdot w$$ for some
$ R_w, S_w\in \ffield.$ We shall show that if $w_0\in W$ is the long
element of the Weyl group, then $R_{w_0}=S_{w_0}.$ By Corollary
\ref{longword_eigenvect_determines}, this suffices to prove the
theorem.

The long coefficient $R_{w_0}$ of $\tDelta \Dem _{w_0}$ is easily read
off from Theorem \ref{thm:long_word}:
\begin{equation}\label{eq:R_{w_0}formula}
 R_{w_0}=\sgn(w_0)\cdot \prod_{\alpha \in \Phi ^+} 
\frac{\left(1-v\cdot x^{m(\alpha )\alpha }\right)\cdot 
x^{m(\alpha )\alpha }}{\left(1-x^{m(\alpha )\alpha }\right)}.
\end{equation}

To determine the coefficient $S_{w_0}$ we again use the property of
$W$-intertwining maps from Corollary \ref{corollary:rearrangepolys} and
argue as in the proof of Lemma \ref{lemma:Rw}.  First, note that the
only term in $\Ts =\sum _{w\in W} T_w=\sum _{w\in W} S_w\cdot w$ that
contributes to the coefficient $S_{w_0}$ is $T_{w_0}.$ (All the other
$T_w$ have fewer than $l(w_0)$ simple reflections appearing in them.)
To examine $T_{w_0},$ fix a reduced expression for the long element: 
$w_0=\sigma _{i_1}\sigma _{i_2} \cdots
\sigma _{i_N}$. Let us again write
\begin{equation*}
  \T_i=q_1(\alpha_i)-q_2(\alpha_i)\sigma_i
\end{equation*}
where $q_1,q_2\colon \Phi\to \ffield _0$ are defined in (\ref{def:q1q2}).
It is clear that the map $q_2$ is $W$-intertwining.  The only
contribution to $S_{w_0}$ from $T_{w_0}=\T _{i_1}\T_{i_2}\cdots
\T_{i_N}$ is from
\[
q_2(\alpha_{i_1})\sigma _{i_1}\cdot q_{2}(\alpha_{i_2})\sigma _{i_2}\cdots 
q_{2}(\alpha_{i_N})\sigma _{i_N}.
\]
Using  Corollary \ref{corollary:rearrangepolys}, we conclude
\begin{equation}
  \label{eq:S_{w_0}formula}
S_{w_0}=\sgn(w_0)\cdot \prod _{\alpha \in \Phi ^+}q_{2}(\alpha) 
=\sgn (w_0)\cdot \prod _{\alpha \in \Phi ^+} 
\frac{\left(1-v\cdot x^{m(\alpha )\alpha }\right)\cdot 
x^{m(\alpha )\alpha }}
{1-x^{m(\alpha )\alpha }}.
\end{equation}
Comparing \eqref{eq:R_{w_0}formula} and \eqref{eq:S_{w_0}formula} we
see that indeed $R_{w_0}=S_{w_0}$, as desired.  This completes the
proof of Theorem \ref{thm:T_sum}.

\section{Whittaker functions}\label{sec:whittaker}

We conclude this paper by showing how to compute Whittaker functions on
certain metaplectic groups using the Demazure and Demazure-Lusztig operators.
This follows from results in \cite[Section 15]{mcnamara2}, which
relate Whittaker functions to the local factors of Weyl group multiple
Dirichlet series constructed in \cite{cg-jams}.  Since these factors
can be constructed using Theorems \ref{thm:long_word} or
\ref{thm:T_sum}, we obtain an alternative description of the Whittaker
functions in the spirit of Demazure's character formula.

Before we can define the Whittaker function of interest, we must
introduce notation and quickly recall the construction of unramified
principal series presentations on metaplectic groups.  Our
presentation is taken from \cite{mcnamara, mcnamara2}, and the reader
should look there for more details.

Let $F$ be a local field containing the $n^{th}$ roots of unity,
$\mu_n$.  We choose once and for all an identification of $\mu_n$ with
the complex $n^{th}$ roots of unity.  Let $\OO$ denote the ring of
integers and $\p$ the maximal ideal of $\OO$ with uniformizer
$\varpi$.  Let $q$ denote the order of the residue field $\OO/\p.$ We
assume that $q\equiv 1 \mod 2n$, so that in fact $F$ contains the $2n$-th
roots of unity.  

In order to define {\em Gauss sums}, we introduce $\psi_F$ be an
additive character on $F$ with conductor $\OO$.  Further let 
$(\,,\,)=(\,,\,)_{F,n}:F^\times \times F^\times \to \mu_n(F)$ be the
$n$th order Hilbert symbol. It is a bilinear form on $F^\times$ that
defines a nondegenerate bilinear form on $F^\times/F^{\times n}$ and satisfies
\[
(x,-x)=(x,y)(y,x)=1,\,x,\,y \in F^\times.
\]
Our assumption that $-1$ is an $n^{th}$ root of unity further implies
that $(\varpi, -1)=1$.
Then we define 
\begin{equation}
  \label{def:gauss_sums}
  g_i=\sum_{u \in \OO^\times/(1+\p)}
    (u,\varpi^i)\psi_F(-\varpi^{-1}u).
\end{equation}
In particular $g_i$ depends only on the residue class of $i$ mod $n$,
$g_0=-1$ and $g_ig_{n-i}=q$.

Now let $G$ be a connected reductive group over $F$.  We assume that
$G$ is split and unramified and arises as the special fiber of a
group scheme $\bG$ defined over $\Z $.  Let $K = \bG(\OO)$ be a
maximal compact subgroup.  Let $T$ be a maximal split torus and let $\Lambda
$ be its group of cocharacters.  Let $B$ be a Borel subgroup
containing $T$, let $U$ be the unipotent radical of $B$, and let
$U^{-}$ be the opposite subgroup to $U$.  Let $\Phi$ be the roots of
$T$ in $G$, and let $\Delta \subset \Phi$ be the simple roots.  The
Weyl group $W$ of $\Phi$ acts on $\Lambda $, and as in Section
\ref{sec:notation} we fix a $W$-invariant integer-valued quadratic form
on $\Lambda$, and use it to define the sublattice $\Lambda_{0}\subset
\Lambda$ as in \eqref{eqn:defoflambda0}.

Let $\tilde G$ be an $n$-fold metaplectic cover of $G$, as defined in
\cite[Section 2]{mcnamara2}.  Thus we have an exact sequence
\begin{equation}\label{eq:cover}
1\rightarrow \mu_{n} \rightarrow \tilde{G} \rightarrow G \rightarrow 1,
\end{equation}
where $\mu_{n}$ is the group of $n$th roots of unity. We choose an
identification of $\mu_n$ with the complex $n^{th}$ roots of unity.
Denote the inverse image of any subgroup $J\subset G$ with a tilde:
$\tilde{J}$.  It is known that \eqref{eq:cover} splits canonically
over $U$ and $U^{-}$.  In general \eqref{eq:cover} does not split over
$K$, but our assumption on $q$ implies that it does.  We therefore fix
a splitting $\tilde{K}\simeq \mu_{n}\times K$ and identify $K$ with
its image in $\tilde{G}$.  Let $H$ be the centralizer in $\tilde{T}$
of $T\cap K$.  The lattice $\Lambda$ (respectively $\Lambda_{0}$) can
be identified with $\tilde{T}/ (\mu_{n}\times (T\cap K))$ (resp., $H/
(\mu_{n}\times (H\cap K))$).  Our assumptions on $G$ imply that $H$ is
abelian, and in fact $H/ (T\cap K) \simeq \mu_{n}\times \Lambda_{0}$
(although not canonically).  Moreover, we may choose a lift of
$\Lambda $ into $\tilde{G}$; we denote this lift by $\lambda \mapsto
\varpi^{\lambda}$.

The unramified principal series representations are parametrized by
complex-valued characters $\chi$ of $\Lambda_{0}$.  Given such a
character, we obtain a character of $H$ using the surjection $H
\rightarrow \mu_{n}\times \Lambda_{0}$, where we let the roots of
unity act faithfully.  We induce this character to $\tilde{T}$ and
obtain a representation $(\pi_{\chi}, i (\chi))$.  The \emph{unramfied
principal series} representation $(\pi_{\chi}, I (\chi))$ is formed
using normalized induction of this representation to $\tilde{G}$.
More precisely, we have
\[
I (\chi) = \{f\colon \tilde{G}\rightarrow i (\chi) : f
(bg)=\delta^{1/2} (b)\pi_{\chi} (b)f (g), b\in \tilde{B}, g\in
\tilde{G}, \text{$f$ locally constant} \},
\]
where $\delta$ is the modular quasicharacter of $\tilde{B}$, and where
$\tilde{G}$ acts on $I (\chi)$ by right translation.
One proves that $I (\chi)^{K}$ is one-dimensional; a nonzero element
$\phi_{K}$ in this space of invariants is called a \emph{spherical
vector}. 

Let $\psi \colon U^{-}\rightarrow \C$ be an unramified character.  By
definition this means that the restriction of $\psi$ to each of the
root subgroups $U_{-\alpha }$, $\alpha \in \Delta$ is a character of
$U_{-\alpha} \simeq F$ with conductor $\OO$.
Then the function $\tilde{G}\rightarrow i (\chi )$ defined by
\begin{equation}\label{eq:ichiwhit}
g\longmapsto \int_{U^{-}}\phi_{K} (ug)\psi (u)du
\end{equation}
is the \emph{$i (\chi)$-valued Whittaker function with character
$\psi$}.  We will obtain a complex-valued Whittaker function by
applying a linear functional $\xi \in i (\chi)^{*}$ to the right of
\eqref{eq:ichiwhit}. We now explain how to construct certain
functionals so that we can arrive at a very explict formula.  To do
this we must be very careful about normalizations.

Recall that $\phi_{K}\in I (\chi)^{K}$ is our spherical vector.  It
turns out that we have an isomorphism $I (\chi)^{K} \simeq i
(\chi)^{\tilde{T}\cap K}$ given by $f\mapsto f (1)$.  Let
$v_{0}=\phi_{K} (1)$.  Let $A$ be a set of coset representatives for
$\tilde{T}/H$; our assumptions imply that we can assume they each have
the form $\varpi^{\lambda}$ for some $\lambda \in \Lambda$.  The
vectors $\{\pi_{\chi} (a)v_{0} : a\in A\}$ give a basis of $i
(\chi)$. 

Now let $\tilde{\chi}\colon \tilde{T}\rightarrow  \C$ be an extension
of $\chi$ to $\tilde{T}$ satisfying $\tilde{\chi} (th) = \tilde{\chi}
(t)\chi (h)$ for all $t\in \tilde{T}, h\in H$.   Such an extension
determines a functional $\xi _{\tilde{\chi}} \in i (\chi)^{*}$ by 
\[
\xi _{\tilde{\chi}} (\pi_{\chi} (a)v_{0}) = \tilde{\chi} (a). 
\]
Since each $a$ has the form $\varpi^{\lambda}$ for $\lambda \in
\Lambda$, we may write instead $\tilde{\chi} (\lambda)$ for
$\tilde{\chi} (a)$.  Then the complex-valued Whittaker function we want to compute is
\begin{equation}\label{eq:Cwhit}
\W = \W_{\tilde{\chi}} : g \longmapsto \xi_{\tilde{\chi}}\Bigl(
\int_{U^{-}}\phi_{K} (ug)\psi (u)du\Bigr).
\end{equation}
The fact that  $\W$ satisfies 
 \[
\W (\zeta ugk) = \zeta \psi (u)\W (g), \quad \zeta \in \mu_{n}, u\in
U, g \in \tilde{G}, k\in K
\]
together with the Iwasawa decomposition $G=UTK$ implies that it
suffices to compute $\W$ on $\tilde{T}$.

We are almost ready to evaluate $\W$ on $\tilde{T}$ in terms of our
Demazure operators.  Set $v=q^{-1}$ in the group action
(\ref{def:action}) and interpret the Weyl group action on $\Lambda$ as
acting on $\tilde{\chi}$ via the identification
$\tilde\chi(\varpi^\lambda)=x^\lambda$.  Further define
\begin{equation}
  \label{eq:cw0}
  c_{w_0}(x)=\frac{\prod_{\alpha\in
      \Phi^+}(1-q^{-1}x^{m(\alpha)\alpha})}
{\prod_{\alpha\in \Phi^+}(1-x^{m(\alpha)\alpha})}
\end{equation}
Then we have the following formula of McNamara:

\begin{theorem}
  \label{thm:mcnamara}\cite[Theorem 15.2]{mcnamara2}  
Let $\lambda$ be a dominant
  coweight.   Then  
\[
(\delta^{-1/2}\W_{\tilde{\chi}}) (\varpi^{\lambda}) = 
 c_{w_0}(x)
\sum_{w\in W} \sgn(w)
\prod_{\alpha \in \Phi(w^{-1})} x^{m(\alpha )\alpha }
w(x^{w_0\lambda}),
\]
where $w$ acts on $x^{\lambda}$ as in \eqref{def:action}.
\end{theorem}

Actually \cite[Theorem 15.2]{mcnamara2} is written in terms of a
slightly different group action introduced in \cite{cg-jams}, but
relating the two actions leads to the statement above.  Combining the
previous result with Theorems \ref{thm:long_word} and \ref{thm:T_sum}
we arrive at our objective of expressing the Whittaker function in
terms of the Demazure and Demazure-Lusztig operators.

\begin{theorem}\label{thm:demazure_whittaker}
  For $\lambda$ a dominant coweight, 
  \begin{align*}
    (\delta^{-1/2}\W_{\tilde{\chi}}) (\varpi^{\lambda}) &= 
\prod_{\alpha\in \Phi^+}(1-q^{-1}x^{m(\alpha)\alpha})
\Dem _{w_0}(x^{w_0\lambda})\\
&= \sum_{w\in W} \T_w(x^{w_0\lambda}).
  \end{align*}
\end{theorem}

\def\cprime{$'$}
\providecommand{\bysame}{\leavevmode\hbox to3em{\hrulefill}\thinspace}
\providecommand{\MR}{\relax\ifhmode\unskip\space\fi MR }
\providecommand{\MRhref}[2]{%
  \href{http://www.ams.org/mathscinet-getitem?mr=#1}{#2}
}
\providecommand{\href}[2]{#2}

\end{document}